\documentclass[10pt]{article}

\usepackage[utf8]{inputenc}
\usepackage[english]{babel}
 
\usepackage[square,numbers]{natbib}
\usepackage{amsmath, amssymb, amsfonts, amsthm , enumitem, hyperref, mathrsfs, standalone, bbm, mathtools, makeidx, color,tikz, amscd}

\usepackage[nottoc]{tocbibind}

\usepackage[all,cmtip]{xy}

\DeclareMathOperator{\tr}{TR}

\DeclareMathOperator{\im}{Im}
\DeclareMathOperator{\Ker}{Ker}
\DeclareMathOperator{\GL}{GL}
\DeclareMathOperator{\proj}{Proj}

\DeclareMathOperator{\HC}{HC}

\newcommand{\T}[1]{\textup{#1}}

\newcommand{\lrar}{\Leftrightarrow}
\newcommand{\rar}{\Rightarrow}

\DeclareMathOperator{\idem}{Idem}

\DeclareMathOperator{\alg}{alg}

\DeclareMathOperator{\rel}{rel}
\DeclareMathOperator{\ch}{ch}
\DeclareMathOperator{\topp}{top}

\newcommand{\C}{\mathbb C}

\newcommand{\N}{\mathbb N}
\newcommand{\R}{\mathbb R}

\newcommand{\scr}[1]{\mathscr #1}
\newcommand{\bb}[1]{\mathbb #1}

\newcommand{\arrm}[3]{\xymatrix{ #2 \ar@{>->}[r]^-{#1} & #3}}
\newcommand{\arre}[3]{\xymatrix{ #2 \ar@{->>}[r]^-{#1} & #3}}

\swapnumbers

\newtheorem{theorem}{Theorem}[section]
\newtheorem{lemma}[theorem]{Lemma}
\newtheorem{prop}[theorem]{Proposition}

\theoremstyle{definition}
\newtheorem{defn}[theorem]{Definition}
\newtheorem{example}[theorem]{Example}

\newtheorem{remark}[theorem]{Remark}

\newtheoremstyle{note}
  {\topsep}{\topsep}%
  {\normalfont}{}%
  {\bfseries}{}%
  {0.5em}{}%
\theoremstyle{note}

\newtheorem{note}[theorem]{}

\numberwithin{equation}{section}

\newcommand{\one}{\mathbbm 1}

\newcommand{\abs}[1]{\left \lvert#1 \right \rvert}
\newcommand{\norm}[1]{\left \lVert #1 \right \rVert}

\newcommand{\sL}{\mathscr L}

\author{
 Peter Hochs\footnote{University of Adelaide, \texttt{peter.hochs@adelaide.edu.au}}
  \and
  Jens Kaad\footnote{University of Southern Denmark, \texttt{jenskaad@hotmail.com}}
  \and
   Andr\'e Schemaitat\footnote{University of M\"unster, \texttt{a.schemaitat@uni-muenster.de}}
}

\begin{document}

\title{Algebraic $K$-theory and a semi-finite Fuglede-Kadison determinant}

\maketitle

\begin{abstract}
In this paper we apply algebraic $K$-theory techniques to construct a Fuglede-Kadison type determinant for a semi-finite von Neumann algebra equipped with a fixed trace. Our construction is based on the approach to determinants for Banach algebras developed by Skandalis and de la Harpe. This approach can be extended to the semi-finite case since the first topological $K$-group of the trace ideal in a semi-finite von Neumann algebra is trivial. On our way we also improve the methods of Skandalis and de la Harpe by considering relative $K$-groups with respect to an ideal instead of the usual absolute $K$-groups. Our construction recovers the determinant homomorphism introduced by Brown, but all the relevant algebraic properties are automatic due to the algebraic $K$-theory framework.
\end{abstract}

\section{Introduction}
The first instance where one encounters the relationship between algebraic $K$-theory and determinants is in the isomorphism between the first algebraic $K$-group of the complex numbers and the complex multiplicative group. This isomorphism is implemented by the determinant of an invertible matrix. In the present paper we will expand on this relationship in the context of Banach algebras and in particular we will see how to recover the Fuglede-Kadison determinant for semi-finite von Neumann algebras as introduced by Brown, \cite{BRO,FK}. Brown based his construction on ideas of Grothendieck \cite{GRO} and Fack \cite{FACK82, FACK83}, who defined a determinant function as an analogue of the product of the eigenvalues up to a given cutoff.

%via a  cutoff procedure on the spectrum of certain elements of semi-finite von Neumann-algebras.

The main advantage of applying an algebraic $K$-theory approach to determinants is that all the algebraic properties of determinants follow as a direct consequence of the definitions. And moreover, when determinants are interpreted as invariants of algebraic $K$-theory they can be used to detect non-trivial elements in these (in general) rather complicated abelian groups. On the other hand, basing the construction of determinants purely on functional analytic methods, requires a substantial amount of work for proving the main algebraic properties and the more conceptual framework provided by algebraic $K$-theory is entirely lost.

The key property that we investigate in this text is the relationship between the operator trace, the logarithm and the determinant as expressed by the identity:
\[
\log( \T{det}(g)) = \T{Tr}( \log(g) ).
\]
In order to expand on this basic relationship in a $K$-theoretic context one considers a unital Banach algebra $A$ together with the homomorphism
\[
\T{GL}(A) \to \T{GL}^{\T{top}}(A),
\]
where $\T{GL}(A)$ denotes the general linear group (over $A$) equipped with the discrete topology and $\T{GL}^{\T{top}}(A)$ is the same algebraic group but with the topology coming from the unital Banach algebra $A$. Passing to classifying spaces and applying Quillen's plus construction, \cite{QUILLEN}, one obtains a continuous map (which is unique up to homotopy)
\[
\T{BGL}(A)^+ \to \T{BGL}^{\T{top}}(A).
\]
By taking homotopy fibres and homotopy groups this gives rise to a long exact sequence of abelian groups:
\[
\begin{CD}
K_{*+1}^{\T{top}}(A) @>{\partial}>> K_*^{\T{rel}}(A) @>>> K_*^{\T{alg}}(A) @>>> K_*^{\T{top}}(A) 
\end{CD},
\]
which is related to the $SBI$-sequence in continuous cyclic homology by means of Chern characters, resulting in the commutative diagram
\begin{equation}\label{eq:checha}
\begin{CD}
K_{*+1}^{\T{top}}(A) @>{\partial}>> K_*^{\T{rel}}(A) @>>> K_*^{\T{alg}}(A) @>>> K_*^{\T{top}}(A) \\
@V{\T{ch}^{\T{top}}}VV @V{\T{ch}^{\T{rel}}}VV @V{\T{ch}^{\T{alg}}}VV @V{\T{ch}^{\T{top}}}VV \\
HP_{*+1}(A) @>{S}>> HC_{*-1}(A) @>{B}>> HN_*(A) @>{I}>> HP_*(A) 
\end{CD}
\end{equation}
of abelian groups, see \cite{KAR,CONKAR}.

In this paper we focus on the low degree (and more explicit) version of this commutative diagram. More precisely, supposing that the unital Banach algebra $A$ comes equipped with a tracial functional $\tau \colon A \to \C$ one obtains an invariant of the continuous cyclic homology group $HC_0(A)$ and hence by precomposition with the relative Chern character we obtain a homomorphism
\[
\tau \circ \T{ch}^{\T{rel}} \colon K_1^{\T{rel}}(A) \to \C.
\]
Supposing furthermore that $K_1^{\T{top}}(A) = \{0\}$ it follows from the commutative diagram in \eqref{eq:checha} combined with Bott-periodicity in topological $K$-theory that the character $\tau \circ \T{ch}^{\T{rel}}$ induces a homomorphism
\[
\T{det}_{\tau} \colon K_1^{\T{alg}}(A) \to \C / \big( 2 \pi i \cdot \T{Im}( \underline \tau ) \big),
\]
where $\underline \tau \colon K_0^{\T{top}}(A) \to \C$ is the character on even topological $K$-theory induced by our tracial functional. In this way we recover the determinant defined by Skandalis and de la Harpe, see \cite{SKANDHARPE,HARPE}.

We extend this framework for defining determinants by incorporating that the tracial functional $\tau$ might only be defined on an ideal $J$ sitting inside the unital Banach algebra $A$ (where $J$ is not required to be closed in the norm-topology of $A$). In this context we assume that $\tau \colon J \to \C$ is a hyper-trace in the sense that $\tau(ja) = \tau(aj)$ for all $a \in A$, $j \in J$. The correct $K$-groups to consider are then relative versions of relative $K$-theory and algebraic $K$-theory and similarly one considers relative versions of the cyclic homology groups appearing in the $SBI$-sequence (we do not use relative topological $K$-theory because of excision). The idea of applying relative $K$-groups in relation to determinant-type invariants of algebraic $K$-theory was (among other things) developed in the PhD-thesis of the second author, \cite{KAAD}.

In the setting of a semi-finite von Neumann algebra $N$ equipped with a fixed normal, faithful and semi-finite trace $\tau \colon N_+ \to [0,\infty]$ it is relevant to look at the trace ideal 
\[
\sL^1_\tau(N) := \{ x \in N : \tau(\abs x) < \infty\}
\]
sitting inside the von Neumann algebra $N$. Since $K_1^{\T{top}}(\sL^1_\tau(N)) = \{0\}$ and $\T{Im} \big ( \underline{\tau} \colon K_0^{\T{top}}(\sL^1_\tau(N)) \to \C \big ) \subseteq \R$ we obtain an algebraic $K$-theory invariant\footnote{In the main text, we denote this map by $\widetilde{\det_{\tau}}$, and use the notation $\det_{\tau}$ for the composition with the isomorphism $\C/i\R \cong (0,\infty)$ given by $z + i\R \mapsto  e^{(z + \bar z) / 2}$.}
\[
\T{det}_{\tau} \colon K_1^{\T{alg}}( \sL^1_\tau(N),N) \to \C/ i \R,
\]
which recovers the Fuglede-Kadison determinant in the context of semi-finite von Neumann algebras, see \cite{BRO,FK}. We emphasize one more time that all the relevant algebraic properties of this determinant follow immediately from its construction. Moreover, we show that $\T{det}_{\tau}$ is given by the explicit formula
\begin{equation} \label{eq det formula}
\T{det}_{\tau}(g) = \tau( \log(|g|))  + i \bb R \quad \big  (g \in \T{GL}_n(N) \, , \, \, g - \one_n \in M_n( \sL^1_\tau(N) \big )
\end{equation}
 Here, $\tau$ is extended to $M_n(N)$ in the obvious way by  taking the sum over the diagonal.

Recently,  the Fuglede-Kadison determinant was generalized in another direction by Dykema, Sukochev and Zanin, to operator  bimodules over II$_1$-factors   \cite{DSZ}. They define this determinant using functional analytic methods,  via an expression analogous to \eqref{eq det formula}. It then requires an elaborate argument to prove that this determinant is multiplicative \cite[Theorem 1.3]{DSZ}.

The present paper is organized as follows: In Section \ref{s:relpai} we introduce the relevant $K$-groups and in Section \ref{s:comseq} we develop the low degree version of the long exact sequence which compares relative algebraic $K$-theory to topological $K$-theory. In Section \ref{section:chern-character} we introduce the low degree version of the relative Chern character in the presence of an ideal $J \subseteq A$. In Section \ref{s:reldet} we present our relative approach to the construction of Skandalis-de la Harpe determinants. In Section \ref{s:toptra} we show that the first topological $K$-group of the trace ideal in a semi-finite von Neumann algebra is trivial and in Section \ref{s:fugkad} we apply this fact to construct the semi-finite Fuglede-Kadison determinant.

\subsection*{Acknowledgements}

The authors thank Fedor Sukochev and Ken Dykema for pointing out some relevant literature.
PH was supported by the European Union, through Marie Curie Fellowship PIOF-GA-2011-299300. JK was supported by the Radboud Excellence Initiative.

\section{$K$-theory for relative pairs of Banach algebras}\label{s:relpai}

\begin{defn}
	\label{def:relative-pairs}
	Let $(A, \norm \cdot _A)$ be a unital Banach algebra and $J \subset A$ be a (not necessarily closed) ideal. We call $(J,A)$ a \textbf{relative pair of Banach algebras} if and only if the following holds:
	\begin{enumerate}
		\item $J$ is a Banach algebra in its own right. Thus, $J$ is endowed with a norm $\norm \cdot _J \colon J \to [0,\infty)$ such that $(J, \norm \cdot _J)$ is a Banach algebra.
		\item For all $a,b \in A$ and $j \in J$ we have
		\[
			\norm{ajb}_J \leq \norm a_A \norm j_J \norm b _A.
		\]
		\item For all $j \in J$ we have
		\[
			\norm j_A \leq \norm j_J.
		\]
	\end{enumerate}
	This defines a category having as objects relative pairs of Banach algebras and morphisms 
	\[
		\phi \colon (J,A) \to (I,B)
	\]
	unital homomorphisms of Banach algebras (not necessarily isometric) $\phi \colon A \to B$ such that $\phi(J) \subset I$ and the restriction $ \phi|_J \colon (J, \norm \cdot _J) \to (I, \norm \cdot _I)$ is bounded.
\end{defn}

\begin{note}
\label{note:relative-pairs}
For a relative pair of Banach algebras $(J,A)$ we obtain    for all $n \in \N$ a relative pair of Banach algebras $(M_n(J), M_n(A))$, where the $(n \times n)$-matrices are equipped with the norm $\norm j _{M_n(J)} := \sum_{k,l = 1}^n \| j_{kl} \|_J$ (and similarly for $M_n(A)$).
\end{note}

	\begin{example}
		Let $H$ be a separable Hilbert space. The standard example of a relative pair is
		\[
			(\scr L^1(H),\sL(H))
		\] 
		where $\scr L^1(H)$ denotes the trace class operators inside the bounded operators $\sL(H)$,  \cite[Th. 2.4.15]{MU}.
	\end{example}
	
	\begin{note}
		The rest of this section is a reminder on various $K$-groups for relative pairs of Banach algebras. A standard reference for topological $K$-theory is \cite{BA}. A very good treatment of algebraic $K$-theory can be found in \cite{ROS}. The probably more uncommon relative $K$-theory of Banach algebras has been introduced in \cite{KAR} and \cite{CONKAR}.
	\end{note}

\begin{defn}
	For a complex algebra $A$ we use the following notations:
	\begin{itemize}
		\item If $A$ is unital, we let $\GL_1(A)$ be the group of invertibles in $A$.
		\item By $\idem(A) = \{e \in A : e^2 = e\}$ we denote the set of idempotents in $A$.
		\item If $A$ is a $*$-algebra, we let $\proj(A) = \{ p \in A : p^2 = p = p^*\}$ be the set of projections in $A$.
	\end{itemize}
\end{defn}

\begin{defn}
	Let $A$ be a Banach algebra. If $A$ has a unit, we denote the group of invertible elements in $M_n(A)$ by $\GL_n(A)$. If $A$ has no unit, we define  for all $n \in \N$ the group
	\[
		\GL_n(A) := \{g \in \GL_n(A^+) : g - \one_n \in M_n(A) \} \subset \GL_n(A^+),
	\]
	where $A^+$ is the unitization of $A$ and $\one_n$ the unit of $\GL_n(A^+)$. Since $M_n(A)^+$ is not the same as $M_n(A^+)$ it is worth noting that  $\GL_1(M_n(A)) \cong \GL_n(A)$, even if $A$ is not unital. This is due to the normalization of the scalar part. We can now define the direct limits
	\begin{itemize} 
		\item $M_\infty(A) := \lim_{n \to \infty} M_n(A)$, 
	under the inclusions 
		\[
			M_n(A) \to M_m(A); \qquad  x \mapsto x \oplus 0_{m-n} \qquad (n \leq m)
	\]
	\item $\GL_\infty(A) := \lim_{n \to \infty} \GL_n(A)$, under the inclusions 
	\[
		\GL_n(A) \to \GL_m(A); \qquad g \mapsto g \oplus \one_{m-n} \qquad (n \leq m)
	\]
	\item $\idem_\infty(A) := \idem(M_\infty(A))$.
	\item $\proj_\infty(A) := \proj(M_\infty(A))$ if $A$ is a $*$-algebra.
	\end{itemize}
	We make $M_\infty(A)$ into a normed algebra by 
	\begin{equation}\label{eq:norm-Minfty}
		\norm{x}_{M_\infty(A)} := \sum_{i,j = 1}^\infty \norm{x_{ij}}_A.
	\end{equation}	
	Then $M_\infty(A)$ equals the normed inductive limit of the $M_n(A)$.
	The norm induces a metric on $\GL_\infty(A)$ by 
	\[
		d(x,y) := \norm{x-y}_{M_\infty(A)}
	\] 
	and makes it into a topological group. By $\GL_\infty^0(A)$ we will denote the path component of the identity in $\GL_\infty(A)$.
\end{defn}

\begin{defn}
	\label{def:top-K}
	The {\bf topological $K$-groups} of the pair $(J,A)$ can be defined to be the usual topological $K$-groups of $J$, i.e. 
	\[
		K_i^{\T{top}}(J,A) := K_i^{\T{top}}(J) \qquad (i=0,1).
	\]
This is due to the fact that topological $K$-theory satisfies excision, \cite[Th. 5.4.2]{BA}.
	Recall that 
	\[
		K_0^{\T{top}}(J) = \T{Ker}\big( K_0^{\T{top}}(J^+) \to K_0^{\T{top}}(\C) \big),
	\] 
where $K_0^{\T{top}}(J^+)$ is the Grothendieck group of the semi-group 
	\[
		V(J^+) = \idem_\infty(J^+) / \sim_s,
	\]
	and $\sim_s$ is the equivalence relation of being similar. To be more precise: If $e,f \in \idem_\infty(J^+)$ we say that $e$ and $f$ are similar if and only if there exists an element $g$ in $\GL_\infty(J^+)$ with $geg^{-1} = f$. The product is calculated in $M_\infty(J^+)$. Furthermore, we have 
	\[
		K_1^{\T{top}}(J) = \GL_\infty(J) / \GL_\infty^0(J).
	\]
	For our purposes it will be useful to know another realization of $K_0^{\T{top}}(J)$, namely $K_2^{\T{top}}(J)$, which may be (cf. \cite[9.1]{BA}) defined by 
	\[
\begin{split}
		K_2^{\T{top}}(J)  
& = \lim_{n \to \infty} \pi_1(\GL_n(J),\one_n) \\ 
& = \lim_{n \to \infty} \big\{[\gamma] \in C^\infty(S^1,\GL_n(J)) / \sim \ : \   \gamma(1) = \one_n \big\}.
\end{split}
	\]
The equivalence relation $\sim$ is given by smooth basepoint preserving homotopies. This definition coincides with the fundamental group of $\GL_\infty(J)$ whose basepoint is the identity. Under this identification one may show that the usual group operation of concatenation of loops coincides with the pointwise product.
	\par The fact that $K_0^{\T{top}}(J)$ and $K_2^{\T{top}}(J)$ are isomorphic is known as Bott periodicity ,  \cite[Th. 9.2.1]{BA}. An explicit isomorphism is given by 
\[
	\beta_J \colon K_0^{\topp}(J) \to K_2^{\topp}(J); \qquad  [e] - [f] \mapsto [\gamma_e \gamma_f^{-1}],
\]
	where $e,f \in \idem(M_n(J^+))$ satisfy  $e-f \in M_n(J)$. The so-called idempotent loops $\gamma_e$ are defined by 
\[
	\gamma_e(z) := ze + \one_n -e \qquad ( z \in S^1).
\]
\end{defn}

\begin{defn}
	The {\bf first algebraic $K$-theory} of the pair $(J,A)$ is defined by
\[
	K_1^{\alg}(J,A) := \GL_\infty(J) / [\GL_\infty(J),\GL_\infty(A)],
\]
where \[
		[\GL_\infty(J),\GL_\infty(A)] := \langle ghg^{-1}h^{-1} : g \in \GL_\infty(J),h \in \GL_\infty(A) \rangle 
	\] 
is a normal subgroup of $\GL_\infty(J)$. $K_1^{\alg}$ is a functor from relative pairs of Banach algebras to abelian groups.
\end{defn}

\begin{defn}
	\label{def:relative-K-theory}
	Let $A$ be a Banach algebra. For all $n \in \N$, we let $R_n(A)$ denote the group of smooth paths $\sigma \colon [0,1] \to \GL_n(A)$ such that $\sigma(0) = \one_n$. The group operation is given by pointwise multiplication. 
	\par Now, let $(J,A)$ be a relative pair of Banach algebras. From the compatibility of the norms on $J$ and $A$ (see Definition \ref{def:relative-pairs}) it follows that 
	\[
		\sigma \tau \sigma^{-1} \tau^{-1} \in R_n(J)  \qquad ( \sigma \in R_n(J), \tau \in R_n(A)).
	\]
We thus have the normal subgroup
\[
[R_n(J),R_n(A)] := \langle \sigma \tau \sigma^{-1} \tau^{-1} \mid \sigma \in R_n(J) \, , \, \, \tau \in R_n(A) \rangle
\]
of $R_n(J)$. On $R_n(J)$ we may consider the equivalence relation $\sim$ of being homotopic with fixed endpoints through a smooth homotopy. Denote by $q\colon R_n(J) \to R_n(J) / \sim$ the quotient map. 
We define
\[
		F_n(J,A) := q\big( [R_n(J),R_n(A)] \big).
\]
 Since $R_n(J) / \sim$ is still a group under pointwise multiplication and $F_n(J,A)$ is a normal subgroup of $R_n(J) / \sim$, we may define
	\[
		K_1^{\rel}(J,A) := \lim_{n \to \infty}\big( (R_n(J) / \sim) / F_n(J,A) \big),
	\]
where the direct limit is taken over the group homomorphisms induced by the inclusions $\GL_n(J) \to \GL_{m}(J)$, $n \leq m$. In this way we obtain a functor from relative pairs of Banach algebras to abelian groups.
\end{defn}

\begin{lemma}
	\label{lem:commutators}
	Let $(J,A)$ be a relative pair of Banach algebras. Let $g \in \GL_\infty(J)$ and $h \in \GL_\infty(A)$. Then, there exist $g_0 \in \GL_\infty^0(J)$ and $h_0 \in \GL_\infty^0(A)$ such that
	\[
		[g,h] = [g_0,h_0].
	\]
	 In particular, 
	\[
		[\GL_\infty(J),\GL_\infty(A)] \subset \GL_\infty^0(J).
	\]
	This standard argument in $K$-theory, which is a variation of Whitehead's lemma, is an immediate consequence of  \cite[Th. 4.2.9]{WO}.
\end{lemma}	

\section{The comparison sequence}\label{s:comseq}

\begin{defn}
	We define the following group homomorphisms 
	\begin{enumerate}[label=$\bullet$]
		\item $\partial \colon K_2^{\T{top}}(J) \to K_1^{\rel}(J,A); \qquad [\gamma]  \mapsto  [t \mapsto \gamma(e^{2 \pi it})]$
		\item $\theta \colon K_1^{\rel}(J,A) \to K_1^{\alg}(J,A); \qquad [\sigma] \mapsto [\sigma(1)^{-1}]$
		\item $p \colon K_1^{\alg}(J,A) \to K_1^{\T{top}}(J); \qquad [g] \mapsto [g]$
	\end{enumerate}
	It is easy to see that $\partial$ and $\theta$ are well-defined. By Lemma \ref{lem:commutators} also $p$ is well defined.
\end{defn}

\begin{lemma}
		\label{lem:comparison-sequence}
		The  sequence 
			\[
			\xymatrix{
			K_2^{\T{top}}(J) \ar[r]^-{\partial} & K_1^{\rel}(J,A) \ar[r]^-\theta & K_1^{\alg}(J,A) \ar[r]^-p & K_1^{\T{top}}(J) \ar[r] & 0
			}
		\]
	is exact.
	\end{lemma}
	
	\begin{proof}
			The only non-trivial thing to check is exactness at $K_1^{\rel}(J,A)$. It is clear that $\theta \circ \partial = 0$. On the other hand, let  $\sigma \in R_n(J)$ and suppose that $[\sigma(1)^{-1}]$ is trivial in  $K_1^{\alg}(J,A)$. Then there are $g_i \in \GL_\infty(J)$ and $h_i \in \GL_\infty(A)$ such that 
			\[
				\sigma(1)^{-1} = \prod_{i=1}^n [g_i,h_i].
			\]
		By Lemma \ref{lem:commutators} we may assume that $g_i \in \GL_\infty^0(J)$ and $h_i \in \GL_\infty^0(A)$.  Hence, for some large enough $m \in \N$ there are smooth paths $\alpha_i \in R_m(J)$ and $\beta_i \in R_m(A)$ connecting $\one_m$ and $g_i$ resp. $h_i$. Then 
		\[
			\tau := \prod_{i=1}^n [\alpha_i,\beta_i] \in [R_m(J), R_m(A)]
		\]
is a path from $\one_m$ to $\sigma(1)^{-1}$. Hence $\gamma := \sigma \cdot \tau^{-1}$ is a smooth loop at $\one_m$ and $\partial([\gamma]) = [\sigma]$ since $[\tau^{-1}]$ is trivial in $K_1^{\rel}(J,A)$.
	\end{proof}
	
	\section{The relative Chern character}
	\label{section:chern-character}	
	
	\begin{note} 
		Let $(J,A)$ be a relative pair of Banach algebras. By $J \otimes_\pi A$ we denote the projective tensor product of $J$ and $A$. The compatibility of the norms on $J$ and $A$ ensures that the multiplication operator
	\[
		m \colon J \otimes_\pi A \to J; \qquad j \otimes a \mapsto   ja
	\]
is bounded.
	\end{note}

	\begin{defn}
		We define the Hochschild boundary map 
		\[
			 b \colon J \otimes_\pi A \to  J; \qquad j \otimes a \mapsto ja - aj
		\]
		and the zeroth relative continuous cyclic homology of the pair $(J,A)$ by 
		\[
			\HC_0(J,A) := J / \im (b).
		\]
		Since $\im (b) \subset J$ might not be closed we regard $\HC_0(J,A)$ simply as a vector space without further topological structure. 
	\end{defn}

\begin{defn} 
\label{def:generalized-trace}
Recall from \ref{note:relative-pairs} that $(M_n(J),M_n(A))$ is a relative pair of Banach algebras for all $n \in \N$. We thus have for each $n \in \N$ the relative continuous cyclic homology groups $\HC_0( M_n(J),M_n(A))$, and we may consider the direct limit of vector spaces
\[
\lim_{n \to \infty} \HC_0( M_n(J),M_n(A)),
\]
where the vector space homomorphisms appearing are induced by the inclusions  $M_n(J) \to M_m(J)$, $n \leq m$. This direct limit is linked to the relative continuous cyclic homology group $\HC_0(J,A)$ via the linear map
\[
\tr \colon \lim_{n \to \infty} \HC_0( M_n(J),M_n(A)) \to \HC_0(J,A),
\]
which is induced by the ``trace'' $\tr \colon M_n(J) \to J$ mapping a matrix to the sum of its diagonal entries. To verify that $\tr$ is indeed well-defined at the level of relative continuous cyclic homology one may translate the proof of \cite[Cor. 1.2.3]{LOD} to our current setting.
\end{defn}

\begin{note}
	Our next task is to construct the relative Chern character. This will be a group homomorphism
	\[
		\ch^{\rel} \colon K_1^{\rel}(J,A) \to \HC_0(J,A) 
	\]
	induced by
	\[
		R_n(J) \ni \sigma \mapsto \tr \left  (\int_0^1 \frac{d \sigma}{dt} \sigma^{-1} \ dt \right ) \in J.
	\]
We shall express $\ch^{\rel}$ as the composition of two homomorphisms: a generalized logarithm 
\[
	\log \colon K_1^{\rel}(J,A) \to \lim_{n \to \infty} \HC_0( M_n(J),M_n(A))
\] and the generalized trace  as defined in \ref{def:generalized-trace}. We now introduce the generalized logarithm:
\end{note}
	
	\begin{prop}
	There is a  well-defined homomorphism
	\[
\log \colon \left \{ \begin{split}
&  K_1^{\rel}(J,A) \to \lim_{n \to \infty} \HC_0( M_n(J),M_n(A)) \\
&  [\sigma] \mapsto \left[ \int_0^1 \frac{d\sigma}{dt} \sigma^{-1} \ dt \right]
\end{split} \right .
	\]
\end{prop}
\begin{proof}
\begin{enumerate}[label=\textbf{(\Roman*)}]
\item Suppose first that $\sigma_0, \sigma_1 \in R_n(J)$ are homotopic through a smooth homotopy $H \colon [0,1] \times [0,1] \to \GL_n(J)$ with fixed endpoints. Thus, $H(t,j) = \sigma_j(t)$ for $j=0,1$.  

We will show that
	\[
		\int_0^1 \frac{d\sigma_1}{dt} \sigma_1^{-1} \ dt  - \int_0^1 \frac{d\sigma_0}{dt} \sigma_0^{-1}  \ dt  \in \im(b),
	\]
where $b \colon M_n(J) \otimes_\pi M_n(A) \to M_n(J)$ is the Hochschild boundary map associated to the relative pair $(M_n(J),M_n(A))$. 
%
%	\item First, we prove the result for paths $\sigma_i \colon [0,1] \to \GL_1(J)$. Let $H \colon [0,1] \times [0,1] \to \GL_1(J) \colon (t,s) \mapsto H(t,s)$ be a homotopy with fixed endpoints between $\sigma_0$ and $\sigma_1$. 
%

Define 
	\[
		L(H) := - \int_0^1 \int_0^1 \frac{\partial H}{\partial t} H^{-1} \otimes \frac{\partial H}{\partial s} H^{-1} \ dt \ ds.
	\]
	We consider $L(H)$ as an element of $M_n(J) \otimes_\pi M_n(A)$ (in fact we even end up in $M_n(J) \otimes_\pi M_n(J)$ which we may then map to $M_n(J) \otimes_\pi M_n(A)$ via the inclusion $M_n(J) \to M_n(A)$). Applying the Hochschild boundary $b$ we see that 
	\[
		b(L(H)) = - \int_0^1 \int_0^1 \left [ \frac{\partial H}{\partial t} H^{-1}, \frac{\partial H }{\partial s} H^{-1} \right ] \ dt \ ds.
	\]
	An easy calculation shows that
	\begin{align*}
		\left [ \frac{\partial H}{\partial t} H^{-1}, \frac{\partial H }{\partial s} H^{-1} \right ]  & = -\frac{\partial H}{\partial t} \frac{\partial H^{-1}}{\partial s} + \frac{\partial H}{\partial s} \frac{\partial H^{-1}}{\partial t}   \\
		&  = \frac{\partial}{\partial t} \left ( \frac{\partial H}{\partial s} H^{-1} \right ) - \frac{\partial}{\partial s} \left ( \frac{\partial H}{\partial t}  H^{-1} \right )  .
	\end{align*}
	 By the fundamental theorem of calculus, we conclude 
	\begin{align*}
		b(L(H)) & = \int_0^1 \int_0^1 \frac{\partial}{\partial s} \left ( \frac{\partial H}{\partial t}  H^{-1} \right )  \ ds \ dt - \int_0^1 \int_0^1 \frac{\partial}{\partial t} \left ( \frac{\partial H}{\partial s} H^{-1} \right ) \ dt \ ds \\
		& = \int_0^1 \left ( \frac{\partial H}{\partial t}(t,1)H(t,1)^{-1} - \frac{\partial H}{\partial s}(t,0)H(t,0)^{-1} \right ) \ dt \\ 
		&   \qquad \qquad - \int_0^1 \left ( \frac{\partial H}{\partial s}(1,s) H(1,s)^{-1} - \frac{\partial H}{\partial s}(0,s)H(0,s)^{-1} \right ) \ ds \\
		 & = \int_0^1 \frac{d \sigma_1}{dt} \sigma_1^{-1} \ dt  - \int_0^1 \frac{d \sigma_0}{dt} \sigma_0^{-1} \ dt.
	\end{align*}
	The second term in the next to last line of our computation vanishes, since our homotopy has fixed endpoints.

We have thus proved that the assignment
\[
\log \colon R_n(J) \to M_n(J); \qquad \sigma \mapsto \int_0^1 \frac{d \sigma}{dt} \sigma^{-1} \ dt
\] 
descends to a well-defined map $\log \colon (R_n(J) / \sim ) \to \HC_0( M_n(J),M_n(A))$. Furthermore, since $\log$ is compatible with direct limits we obtain a well-defined map
\[
\log \colon \lim_{n \to \infty} (R_n(J) / \sim ) \to \lim_{n \to \infty}\HC_0( M_n(J),M_n(A)).
\]
We will now show that $\log ([\sigma_0 \cdot \sigma_1] )= \log ([\sigma_0]) + \log ([\sigma_1])$ for all $\sigma_0, \sigma_1 \in R_n(J)$. Choose a smooth function $\phi \colon \R \to [0,1]$ such that $\phi( (-\infty,0] ) = \{0\}$ and $\phi( [1/2,\infty)) = \{1\}$. Define the smooth function $\psi \colon \R \to [0,1]$ by $\psi(t) := \phi(t - 1/2)$. We then have that 
\[
\sigma_0 \sigma_1 \sim (\sigma_0 \circ \psi) \cdot (\sigma_1 \circ \phi)
\]
and it thus suffices to verify that $\log \big( (\sigma_0 \circ \psi) \cdot (\sigma_1 \circ \phi) \big) = \log(\sigma_0) + \log( \sigma_1)$. But this identity follows by a change of variables:
\[
\begin{split}
& \log \big( (\sigma_0 \circ \psi) \cdot (\sigma_1 \circ \phi) \big) \\
& \quad  = \int_0^{1/2} \frac{d (\sigma_1 \circ \phi)}{dt} (\sigma_1 \circ \phi)^{-1} \ dt
+ \int_{1/2}^1 \frac{d (\sigma_0 \circ \psi)}{dt} (\sigma_0 \circ \psi)^{-1} \ dt \\
& \quad = \log(\sigma_0) + \log(\sigma_1).
\end{split}
\]
\item To finish the proof of the proposition we only need to show that $\log ([ \sigma\tau \sigma^{-1}]) = \log([\tau])$ whenever $\sigma \in R_n(A)$ and $\tau \in R_n(J)$. To this end, we consider the smooth homotopy with fixed endpoints
	\[
		H(s,t):= \sigma(f(s,t))\tau(t)\sigma(f(s,t))^{-1}, \quad f(s,t) := ts + 1- s = s(t-1)  + 1
	\] 
between $\sigma \tau \sigma^{-1}$ and $\sigma(1)\tau \sigma(1)^{-1}$. This proves that
	\[
		\log([\sigma \tau \sigma^{-1}] )
= \log([\sigma(1) \tau \sigma(1)^{-1}])
%= \sigma(1) \cdot \log ([\tau]) \cdot \sigma(1)^{-1}  \\
= \log([\tau]),
\]
where we have used that $\sigma(1) x \sigma(1)^{-1}$ and $x$ determine the same element in $\HC_0(M_n(J),M_n(A))$ for all $x \in M_n(J)$.
\end{enumerate}
\end{proof}

\begin{defn}
\label{def:ch}
By the {\bf relative Chern character} $\ch^{\rel} \colon  K_1^{\rel}(J,A) \to \HC_0(J,A)$ we understand the homomorphism obtained as the composition
\[
\begin{CD}
\ch^{\rel} \colon K_1^{\rel}(J,A) @>{\log}>> \lim_{n \to \infty} \HC_0( M_n(J),M_n(A)) @>{\tr}>> \HC_0(J,A)
\end{CD}
\]
of the generalized logarithm and the generalized trace.
\end{defn}

\section{The relative Skandalis-de la Harpe determinant}\label{s:reldet}

\begin{note}
	Analogous to the determinant of Skandalis and de la Harpe, we are now in a position to define such a determinant purely by means of $K$-theory for relative pairs of Banach algebras. In particular we are able to deal with the presence of a not necessarily closed ideal $J$  inside a unital Banach algebra $A$.
\end{note}

\begin{defn}
	Let $(J,A)$ be a relative pair of Banach algebras. In this section we assume $\tau \colon J \to \C$ to be a continuous linear functional which additionally satisfies 
	$$
		\tau(ja) = \tau(aj) \qquad ( a \in A, j \in J)
	$$
	The latter property means that $\tau$ is a \textbf{hyper-trace}. For such a trace there is a well-defined map  (also denoted by $\tau$)
	$$
		\tau \colon \HC_0(J,A) \to \bb C; \qquad  j + \T{Im} \, (b) \mapsto \tau(j).
	$$
	Furthermore, we let
	$$
		\tilde \tau : = - \tau \circ \ch^{\rel} \colon  K_1^{\rel}(J,A) \to \C ,
	$$ with $\ch^{\rel}$ as in Definition \ref{def:ch}. Note that $\tilde \tau$ is a homomorphism into the additive group $\C$.
\end{defn}

\begin{note}
	Recall  (Lemma \ref{lem:comparison-sequence}) that there is an exact sequence in relative $K$-theory:
	\[
		\xymatrix{
		K_2^{\topp}(J) \ar[r]^\partial  & K_1^{\rel}(J,A) \ar@{->}[r]^{\theta} & K_1^{\alg}(J,A)  \ar[r]^p & K_1^{\topp}(J)  \ar[r] & 0
		}  .
	\]
	This allows us to define the \textbf{relative Skandalis-de la Harpe determinant}
	$$
		\widetilde {\rm det}_\tau \colon \T{Ker} \,(p) \to \bb C / \, \T{Im}(\tilde \tau \circ \partial)
	$$ by 
	$$
		\widetilde {\rm det}_\tau([g]) := \tilde \tau([\sigma]) +  \T{Im}(\tilde \tau \circ \partial) ,
	$$
	where $[\sigma] \in K_1^{\rel}(J,A)$ satisfies $\theta([\sigma]) = [\sigma(1)^{-1}] = [g] \ni K_1^{\alg}(J,A)$. Such a lift always exists since $\T{Ker} \, (p) = \T{Im} \,(\theta)$. Furthermore, this assignment is well-defined since if $[\sigma_0]$ and $[\sigma_1]$ are lifts of the same element $[g]$ then $[\sigma_0][\sigma_1]^{-1} \in \T{Ker} \,( \theta) = \T{Im} \, ( \partial)$. It  follows that $\tilde \tau([\sigma_0]) = \tilde \tau ([\sigma_1])$ modulo $\T{Im} \, (\tilde \tau \circ \partial)$.
	
	Compare this with the definition on page 245 of \cite{SKANDHARPE}, where absolute $K$-theory is used rather than relative $K$-theory. 
	\end{note}
	
	\begin{lemma}
		\label{lem:image-trace}
		We have the following equality of subgroups of $(\C,+)$:
		$$
	 2 \pi i 	\cdot \T{Im}  \big (	\underline \tau \colon K_0^{\topp}(J) \to  \C  \big  ) = \T{Im}(\tilde \tau \circ \partial \colon K_2^{\topp}(J) \to \C ).
		$$
		By $\underline \tau \colon K_0^{\topp}(J) \to  \C$ we understand the map induced by $\tau$.
	\end{lemma}
	
	\begin{proof}
		The claim follows from commutativity of the following diagram:
		\begin{equation} \label{eq diag trace}
			\xymatrix{
				K_0^{\topp}(J) \ar[r]^{\beta_J}_\cong \ar[dr]_{- 2 \pi i \cdot \underline \tau} & K_2^{\topp}(J) \ar[d]^{\tilde \tau \circ \partial} \\
				& \C.
			}
		\end{equation}
		By $\beta_J$ we mean the Bott isomorphism map, as in Definition \ref{def:top-K}. 
		
To prove commutativity of \eqref{eq diag trace}, we note that for  $f \in \idem(M_n(J^+))$,
\[
\ch^{\rel}(\partial([\gamma_f])) = \tr \left ( 2 \pi i  \int_0^1 e^{2\pi i t}f (e^{-2 \pi it}f + \one_n - f) \ dt \right ) = 2 \pi i \tr(f).
\]
If now $e,f \in \idem(M_n(J^+))$ satisfy $e-f \in M_n(J)$, then 
\[
\tilde \tau (\partial ( \beta_J([e] - [f]))) = \tilde \tau ( \partial ([\gamma_e  \gamma_f^{-1}])) = - 2 \pi i \cdot \tau( \tr(e - f) ) .
\]
So \eqref{eq diag trace} indeed commutes.
	\end{proof}
	
	\begin{note}	
		\label{note:determinant}
		 Together with Lemma \ref{lem:image-trace} we see that the following diagram commutes:
	\[
		\xymatrix{
		K_2^{\topp}(J) \ar[r]^\partial \ar[d]^{\tilde \tau \circ \partial} & K_1^{\rel}(J,A) \ar@{->>}[r]^{\theta} \ar[d]^{\tilde \tau} & \Ker p   \ar[d]^{\widetilde{\mathrm{det}_\tau}} \ar[r]^p & 0  \\
		 2 \pi i \cdot  \T{Im} \, (\underline \tau ) \ar@{>->}[r] & \C \ar@{->>}[r] & \C /  \big ( 2 \pi i \cdot  \T{Im} \, (\underline \tau )   \big ) & 
		} 
	\]
	In the next section this will be applied to the case that the kernel of $p$ is all of $K_1^{\alg}(J,A)$. In that case we get a determinant 
	$$
		\widetilde {\rm det}_\tau : K_1^{\alg}(J,A) \to \bb C / \big ( 2 \pi i \cdot \T{Im } (\underline \tau ) \big ) .
	$$
\end{note}

\section{Topological $K$-theory of trace ideals}\label{s:toptra}

\begin{note} 
	In the following $N \subset \scr L(H)$ will always denote a semi-finite von Neumann algebra equipped with a fixed normal, faithful and semi-finite trace $\tau \colon N_+ \to [0,\infty]$. A good reference for traces on von Neumann algebras is \cite[I.6.1,\ I.6.10]{DIX}.
\end{note}

\begin{note}\label{L1tau}
We let $\| \cdot \| \colon N \to [0,\infty)$ denote the operator norm on $N$ and we let
\[
\scr L^1_\tau(N) := \{x \in N : \tau(\abs x) < \infty \} \subset N
\]
denote the trace ideal in $N$. We recall that $\scr L^1_\tau(N) \subset N$ is indeed a $*$-ideal and that $\scr L^1_\tau(N)$ becomes a Banach $*$-algebra in its own right when equipped with the norm
\[
\norm x _{1,\infty} := \norm x + \tau(\abs x)%\norm x_1 
\qquad (x \in \scr L_\tau^1(N))
\]
Moreover, it holds that $(\scr L^1_\tau(N),N)$ is a relative pair of Banach algebras in the sense of Definition \ref{def:relative-pairs}.
\end{note}

\begin{note}\label{MnL1}
For each $n \in \N$ we have that $M_n(N) \subset \scr L( H^{\oplus n})$ is a semi-finite von Neumann algebra. Indeed, we may define the normal, faithful and semi-finite trace $\tau_n \colon M_n(N)_+ \to [0,\infty]$ by $\tau_n(x) := \sum_{i = 1}^n \tau(x_{ii})$. The inclusion $M_n( \scr L^1_\tau(N)) \to M_n(N)$ then induces an isomorphism
\[
M_n(\scr L^1_\tau(N)) \cong \scr L^1_{\tau_n}( M_n(N))
\]
of Banach $*$-algebras. This isomorphism is however not an isometry since $M_n(\scr L^1_\tau(N))$ is (by convention) equipped with the norm defined as in \eqref{eq:norm-Minfty}.
\end{note}

\begin{remark}
	In order to develop the theory, we want to remark that one can drop the assumption of $\tau$ being semi-finite. This property is only (implicitly) used to prove that $K_1^{\topp}(\sL_\tau^1(N))$ is trivial. However, there is an isomorphism of Banach $*$-algebras
	$$
		\sL_\tau^1(N) \cong \sL_\tau^1(pNp),
	$$ 
	where $p = \sup \{q \in \proj(N)  : \tau(q) < \infty \}$. Of course, on the right hand side we consider the restriction of the trace $\tau : N_+ \to [0,\infty]$ to $(pNp)_+ = pN_+p$. But it is clear that $\tau : pN_+p \to  \bb [0,\infty]$ is a semi-finite trace, even if we do not assume that $\tau$ itself is semi-finite.
\end{remark}

\begin{note}
Our main task in this section is to prove that the first topological $K$-theory of the trace ideal of a semi-finite von Neumann algebra is trivial. This will allow us to define a determinant on the first algebraic $K$-theory of the relative pair $(\scr L_\tau^1(N),N)$. The following basic estimates are crucial:

\begin{enumerate}
\item There exist $c,C > 0$ such that 
\[
c |e^{it} - 1| \leq |t| \leq C |e^{it} - 1| \qquad ( t \in [-\pi,\pi])
\]
Thus, for all $x \in N$ with $x = x^*$ and $\| x \| \leq \pi$, it holds that 
\begin{equation}\label{eq:exp-trace-class}
e^{ix} - \one \in \scr L^1_\tau(N) \lrar x \in \scr L^1_\tau(N)
\end{equation}
\item Let $I \subset \R$ be a compact interval. Then there exist $d,D > 0$ such that 
\[
d |e^t - 1| \leq |t| \leq D |e^t - 1| \qquad (t \in I)
\]
Thus, for all $x \in N$ with $x = x^*$, it holds that
\begin{equation}\label{eq:log-trace-class}
e^x - \one \in \scr L^1_\tau(N) \lrar x \in \scr L^1_\tau(N)
\end{equation}
\end{enumerate}
\end{note}

\begin{defn}
For each $n \in \N$ we let $U_n( \scr L^1_\tau(N)) \subset \GL_n( \scr L^1_\tau(N))$ denote the subgroup defined by
\[
U_n( \scr L^1_\tau(N))  := \big\{ u \in \GL_n( \scr L^1_\tau(N)) : u u^* = \one_n = u^* u \big\}
\]
and we equip $U_n(\scr L^1_\tau(N))$ with the metric inherited from $\GL_n(\scr L^1_\tau(N))$.

\end{defn}
	
\begin{lemma}
\label{lem:GL_1-path-connected-semi-finite}
The group $U_n(\scr L_\tau^1(N))$ is path connected for all $n \in \N$.
\end{lemma}
\begin{proof}
Since $M_n( \scr L_\tau^1(N)) \cong \scr L_{\tau_n}^1( M_n(N))$ by \ref{MnL1}, it suffices to verify the lemma for $n = 1$. Let thus $u \in N$ be a unitary with $u - \one \in \scr L_\tau^1(N)$. Since $u$ is unitary, we have that $\sigma(u) \subset S^1$ and we may choose a branch of the logarithm such that $x := -i\log(u)$ satisfies the conditions $\| x \| \leq \pi$ and $x = x^*$. By \eqref{eq:exp-trace-class} we  have that $x \in \scr L_\tau^1(N)$ and hence that the path  $\gamma \colon t \mapsto e^{itx}$ defines a continuous path from $\one$ to $e^{ix} = u$. Indeed, the continuity follows since $\exp \colon \scr L_\tau^1(N) \to \scr L_\tau^1(N)^+$ is continuous (by holomorphic functional calculus).
\end{proof}

\begin{lemma}\label{lem:positive-path-connected}
If $g \in \GL_n(\scr L_\tau^1(N)) \cap M_n(N)_+$, then there exists a continuous path $\gamma \colon [0,1] \to \GL_n(\scr L_\tau^1(N)) \cap M_n(N)_+$ connecting $\one_n$ and $g$.
\end{lemma}
\begin{proof}
This follows since $\exp \colon M_n( \scr L_\tau^1(N)) \cap M_n(N)_{\T{sa}} \to \GL_n(\scr L_\tau^1(N)) \cap M_n(N)_+$ is a bijection by (\ref{eq:log-trace-class}) (where $M_n(N)_{\T{sa}}$ denotes the self-adjoint elements in $M_n(N)$).
\end{proof}

\begin{lemma}\label{lem:polar-decomposition}
For each $n \in \N$ we have that
\[
g \in \GL_n(\scr L_\tau^1(N)) \rar |g| \in \GL_n(\scr L_\tau^1(N)).
\]
\end{lemma}
\begin{proof}
Let $g \in \GL_n(\scr L_\tau^1(N))$. Since 
\[
g^*g - \one_n = g^*(g-\one_n) + (g-\one_n)^*,
\]
we see that $g^*g - \one_n \in M_n\big( \scr L_\tau^1(N) \big)$ and hence that  (by (\ref{eq:log-trace-class})) $\log(g^* g) \in M_n(\scr L_\tau^1(N)) \cap M_n(N)_{\T{sa}}$. But then we have that $e^{ \log(g^* g)/2} = \abs g \in \GL_n(\scr L_\tau^1(N))$, proving the lemma.
\end{proof}

\begin{prop}
\label{prop:K1-trivial}
Let $N$ be a semi-finite von Neumann algebra and let $\tau \colon N_+ \to [0,\infty]$ be a fixed normal, faithful and semi-finite trace. Then, the first topological $K$-group of the Banach algebra $(\mathscr L^1_\tau(N),\norm \cdot _{1,\infty})$ is trivial, i.e.
\[
K_1^{\T{top}}(\mathscr L^1_\tau(N)) = \{0\}.
\]
\end{prop}
\begin{proof}
Let $g \in \GL_n( \scr L^1_\tau(N))$ be given. By Lemma \ref{lem:polar-decomposition} we have that $\abs g \in \GL_n(\scr L^1_\tau(N)) \cap M_n(N)_+$ and hence that $u := g \abs g ^{-1} \in U_n( \scr L^1_\tau(N))$. To see this one can note that $M_n(\sL_\tau^1(N)) \ni g (\abs g - \one_n) \abs g^{-1} = g - u$ and $u - \one_n = -(g-u) + (g-\one_n)$. Since $g = u \abs g$, the result now follows by Lemma \ref{lem:GL_1-path-connected-semi-finite} and Lemma \ref{lem:positive-path-connected}.
\end{proof}

%\begin{lemma}\label{lem:idem-replace}
%Let $e \in \idem_\infty(\scr L_\tau^1(N)^+)$. Then $e$ is similar to a projection in $M_\infty(\scr L_\tau^1(N)^+)$.
%\end{lemma}
%\begin{proof}
%The proof can be found in \cite[Prop. 4.6.2]{BA}.
%\end{proof}

\section{The semi-finite Fuglede-Kadison determinant}\label{s:fugkad}

\begin{note}
	We are now going to use $K$-theory for relative pairs of Banach algebras to define our determinant. From  \ref{L1tau} we know that $(\scr L_\tau^1(N),N)$ is a relative pair of Banach algebras and that $\tau \colon \scr L_\tau^1(N) \to \bb C$ is continuous with respect to $\norm \cdot_{1,\infty}$. Since $\tau \colon \scr L_\tau^1(N) \to \C$ is moreover a hyper-trace we get (as defined in \ref{note:determinant}) a determinant
	$$
		\widetilde {\rm det}_\tau : K_1^{\alg}(\sL_\tau^1(N),N) \to \C / \big (2 \pi i \cdot \T{Im} \, (\underline \tau) \big ).
	$$
	Note that our determinant is defined on all of $K_1^{\alg}(\sL_\tau^1(N),N)$ by Proposition \ref{prop:K1-trivial}, i.e. our map $p : K_1^{\alg}(\sL_\tau^1(N),N) \to K_1^{\topp}(\sL_\tau^1(N))$ is the zero-map. 
\end{note}

\begin{lemma}\label{lem:iR}
We have
\[
\T{Im} \big (\underline \tau : K_0^{\topp}(\sL_\tau^1(N)) \to \C \big ) \subset \R.
\] 
\end{lemma}
\begin{proof}
Every idempotent  in $\idem_\infty(\scr L_\tau^1(N)^+)$ is similar to a projection in $M_\infty(\scr L_\tau^1(N)^+)$ (see \cite[Prop. 4.6.2]{BA}).	
%From Lemma \ref{lem:idem-replace} we know that we may replace idempotents by projections. Thus, 
And for $p,q \in \proj(M_n(\sL_\tau^1(N)^+))$ with $p-q \in M_n(\sL_\tau^1(N))$ we see that 
\[
	\underline \tau([p]-[q]) =  \tau( \tr(p - q) ) \in  \bb R,
\]
where we have used that all the diagonal entries $(p - q)_{jj}$ are self-adjoint.
\end{proof}

\begin{note}
We thus have  a well-defined homomorphism
\[
	\widetilde{\mathrm{det}_\tau} \colon K_1^{\alg}(\sL_\tau^1(N),N) \to \bb C / i \bb R 
\quad \widetilde{\mathrm{det}_\tau} \colon [g] \mapsto \tilde \tau([\sigma]) + i \bb R,
\] 
	where $[\sigma] \in K_1^{\rel}(\scr L_\tau^1(N),N)$ is any lift of $[g]$, by which we mean that $\theta([\sigma]) = [\sigma(1)^{-1}] = [g]$.
	\par 
	Note that there is an isomorphism of abelian groups
	\[
		\bb C / i \bb R \to (0,\infty); \qquad z + i \bb R \mapsto e^{\Re(z)},
	\] where $\Re(z)$ denotes the real part of $z \in \C$. This gives rise to the following definition:
	\end{note}
	
	\begin{defn}
We define the {\bf semi-finite Fuglede-Kadison determinant}
\[
\mathrm{det}_\tau \colon K_1^{\alg}(\sL_\tau^1(N),N) \to (0,\infty)
\]
 by 
\[
	\mathrm{det}_\tau([g]) := e^{\Re(\widetilde{\mathrm{det}_\tau}([g]))}.
\]
	More explicitly, we have
	\[
	\mathrm{det}_\tau([g]) = \exp((\Re \circ \tilde \tau) [\sigma] ) = \exp \left  ( - (\Re \circ \tau \circ \tr) \left ( \int_0^1 \frac{d \sigma} {dt} \sigma^{-1} \ dt \right ) \right ) ,
	\]
	where $[\sigma] \in K_1^{\rel}(\sL_\tau^1(N),N)$ is any lift of $[g] \in K_1^{\alg}(\sL_\tau^1(N),N)$, meaning that $[\sigma(1)^{-1}] = [g]$.
	\end{defn}

\begin{remark}
In his description of the Fuglede-Kadison determinant for finite von Neumann algebras $N$, de la Harpe 
used the standard fact that $\GL_{\infty}(N)$ is connected, i.e.\ $K_1^{\T{top}}(N) = \{0\}$ (see Corollary 14 in \cite{HARPE}).
In the semi-finite case, we need the more subtle Proposition \ref{prop:K1-trivial}.
\end{remark}

\begin{prop}
	\label{cor:properties}
	The semi-finite Fuglede-Kadison determinant $\mathrm{det}_\tau$ has the following properties:
	\begin{enumerate}
		\item $\mathrm{det}_\tau([gh]) = \mathrm{det}_\tau([g]) \mathrm{det}_\tau ([h])$ for all $g,h \in \GL_\infty( \sL_\tau^1(N) )$.
		\item $\mathrm{det}_\tau([hgh^{-1}]) = \mathrm{det}_\tau([g])$ for all $g \in \GL_\infty(\sL_\tau^1(N))$ and $h \in \GL_\infty(N)$.
		\item $\mathrm{det}_\tau([e^x]) = (\exp \circ \Re \circ \tau \circ \tr)(x)$ for $x \in M_\infty(\sL_\tau^1(N))$.
		\item For self-adjoint $x \in \GL_1(\sL_\tau^1(N))$ we get the usual relation
	\[
		\mathrm{det}_\tau([e^x]) = e^{\tau(x)}.
	\]
	\end{enumerate}
\end{prop}
These properties follow directly from the definition of the determinant, and the fact that $\widetilde{\mathrm{det}_\tau}$ is a homomorphism. This is an advantage of the $K$-theoretic approach to constructing determinants over the functional analytic approach.  See e.g.\ Section 1 of \cite{BRO}, where the equality in the following proposition is the definition of the determinant. 
\begin{prop}
	The following explicit formula holds:
	\[
		\mathrm{det}_\tau([g]) = e^{\tau(\log \abs g)} \qquad (g \in \GL_1(\sL_\tau^1(N))).
	\]
\end{prop}	
\begin{proof}
	Let $g \in \GL_1(\scr L_\tau^1(N))$. Using the polar decomposition $g = u \abs g$ and Lemma \ref{lem:polar-decomposition} we may compute
	\[
		\mathrm{det}_\tau([g]) = \mathrm{det}_\tau([u]) \mathrm{det}_\tau([\abs g]).
	\]
	Since $u$ is a unitary in a von Neumann algebra, we may write $u = e^{ia}$ for some self-adjoint $a \in N$. By property 3. of Proposition \ref{cor:properties} we know that  
	\[
	\mathrm{det}_\tau([u]) = e^{\Re(\tau(ia)))} = 1.
\] 
	Using property 4. we see that 
	\[
		\mathrm{det}_\tau([g]) = \mathrm{det}_\tau([\abs g]) = \mathrm{det}_\tau([e^{\log \abs g}]) = e^{\tau( \log \abs g)}.
	\]
\end{proof}

\bibliography{main}

\end{document}